\documentclass[a4paper,11pt]{amsart}

\usepackage{lmodern}

\usepackage[T1]{fontenc}        
\usepackage[utf8]{inputenc}   
\usepackage[english]{babel}
\usepackage{amssymb,amsfonts,amsthm,amsmath,latexsym}
\usepackage{enumerate, caption, array, esint}
\usepackage{hyperref,color}

\usepackage[margin=1.1in]{geometry}

\theoremstyle{plain}
\newtheorem{theorem}{Theorem}

\newtheorem{lemma}{Lemma}
\newtheorem{proposition}{Proposition}

\theoremstyle{definition}
\newtheorem{definition}{Definition}

\theoremstyle{remark}
\newtheorem*{remark}{Remark}
\newtheorem*{remarks}{Remarks}

\newcommand{\ssum}[1]{\sum_{\substack{#1}}}

\newcommand{\ee}{{\varepsilon}}
\newcommand{\e}{{\rm e}}

\renewcommand{\tilde}{\widetilde}

\renewcommand{\bar}{\overline}

\newcommand{\CC}{{\mathbb C}}
\newcommand{\NN}{{\mathbb N}}
\newcommand{\QQ}{{\mathbb Q}}
\newcommand{\RR}{{\mathbb R}}
\newcommand{\ZZ}{{\mathbb Z}}

\newcommand{\cA}{{\mathcal A}}

\newcommand{\cI}{{\mathcal I}}
\newcommand{\cJ}{{\mathcal J}}

\newcommand{\cQ}{{\mathcal Q}}
\newcommand{\cR}{{\mathcal R}}
\newcommand{\cS}{{\mathcal S}}

\DeclareMathOperator{\ind}{\mathbf{1}}

\newcommand{\bb}{{\bf b}}

\newcommand{\floor}[1]{{\left\lfloor {#1} \right\rfloor}}
\renewcommand{\mod}[1]{\ (\text{mod }#1)}

\newcommand{\abs}[1]{\left| #1 \right|}
\newcommand{\norm}[1]{\left\| #1 \right\|}
\newcommand{\rb}[1]{\left( #1 \right)}

\numberwithin{equation}{section}

\title{Exponential sums with automatic sequences}

\author{S. Drappeau}
\address{Aix Marseille Université, CNRS, Centrale Marseille \\ I2M UMR 7373 \\ 13453 Marseille \\France}
\email{sary-aurelien.drappeau@univ-amu.fr}

\author{C. Müllner}
\address{Institut f\"ur Diskrete Mathematik und Geometrie \\ TU Wien \\ Wiedner Hauptstr. 8–10 \\ 1040 Wien, Austria}
\email{clemens.muellner@tuwien.ac.at}

\thanks{Both authors acknowledge support by the project MuDeRa (Multiplicativity, Determinism and Randomness), 
which is a joint project between the ANR (Agence Nationale de la Recherche) and the FWF (Austrian Science Fund).
Moreover, the second author was supported by the Austrian Science Foundation FWF, project F5502-N26, which is a part of the Special Research Program ``Quasi Monte Carlo Methods: 
Theory and Applications''. The authors thank C. Mauduit for helpful discussions.}


\date{\today}

\begin{document}

\begin{abstract}
  We show that automatic sequences are asymptotically orthogonal to periodic exponentials of type~$e_q(f(n))$, where $f$ is a rational fraction, in the P\'olya-Vinogradov range.
  This applies to Kloosterman sums, and may be used to study solubility of congruence equations over automatic sequences.
  We obtain this as consequence of a general result, stating that sums over automatic sequences can be bounded effectively in terms of two-point correlation sums over intervals.
\end{abstract}

\maketitle{}

\section{Introduction}

A complex-valued sequence~$(a_n)$ is called \emph{automatic}, if there is a finite deterministic automaton such that for each~$n$, the value~$a_n$ is given by a function of the final state of the automaton, when the automaton is given as input the digital representation of~$n$. There has been strong interest recently on understanding correlation of automatic sequences with other types of arithmetical functions.
Much of this interest has stemmed from the Sarnak conjecture: it was recently shown by the second author~\cite{mullner} that all automatic sequences are asymptotically orthogonal to the Möbius function~$\mu(n)$, in the sense that~$\sum_{n\leq x} a_n \mu(n)= o(x)$ as~$x\to\infty$.

In the present paper, we are interested in asymptotic orthogonality of automatic sequences with oscillating functions given by periodic exponentials of rational fractions. The prototype of correlations we wish to study are the incomplete Kloosterman sums
\begin{equation}\label{eq:sum-incomplete}
\ssum{n\in \cI \\ (n, q)=1} a_n \e\Big(\frac{\bar{n}}q\Big), \qquad (\e(z)=\e^{2\pi i z},\ n\bar{n} = 1 \mod{q})
\end{equation}
for an interval~$\cI$ of integers. Our goal is to find conditions on~$q$ and on the size of the interval~$|\cI|$ which ensure that we have \emph{asymptotic orthogonality} of~$(a_n)$ with~$(\e(\bar{n}/q))$, in the sense that the sum in~\eqref{eq:sum-incomplete} is~$o(|\cI|)$ as~$|\cI|\to\infty$.

When~$(a_n)$ is constant, a classical result of Weil~\cite{Weil} shows that the condition~$|\cI| \geq q^{1/2 + \ee}$ suffices: 
we will refer to this condition as the \emph{P\'olya-Vinogradov range} 
(in reference to the P\'olya-Vinogradov bound for sums of Dirichlet characters). 
This may be improved in specific circumstances~\cite{Korolev, Irving}, however, the range obtained by the Weil bound remains unsurpassed in general.

Our main result, which we will describe shortly, shows that asymptotic orthogonality for~\eqref{eq:sum-incomplete} holds in the P\'olya-Vinogradov range for \emph{all} automatic sequences.

\subsection*{Statement of results}

Let us now describe the precise setting of our study. Given~$k\geq 2$ a base, $\Sigma = \{0, \dotsc, k-1\}$, $\cA = (Q, \Sigma, \delta, q_0, \tau)$ a deterministic finite automaton with output function~$\tau:Q\to \CC$, we define the associated automatic sequence~$(a_n)_{n\geq 0} = (\tau(\delta(q_0,(n)_k)))_{n\geq 0}$, where $(n)_k$ denotes
the representation of $n$ in base $k$ without leading zeros. When we refer to an \emph{automatic sequence} in what follows, it will always be one given by such a construction. 
In particular, we assume without loss of generality that $\delta(q_0,0) = q_0$.
For a more detailed treatment of automatic sequences see for example \cite{alloucheShallit}.

Given a rational fraction~$f = P(X)/Q(X) \in\QQ(X)$, $n\in\ZZ$ and~$q\in\NN_{>0}$, we define~$\e_q(f(n))$ following the definition of section~4A of~\cite{Polymath}; we describe this in detail below (Section~\ref{sec:weil-bounds}) and simply note for now that whenever~$(Q(n), q)=1$, we have
$$ \e_q(f(n)) = \e\Big(\frac{P(n)\bar{Q(n)}}q\Big). $$

\begin{definition}
Let~$f\in\QQ(X)$, which we write in reduced form~$f=P/Q$ with~$P, Q\in\ZZ[X]$ and coprime. Let also an integer~$q\geq 1$ be given.
\begin{enumerate}[(i)]
\item We denote by~$(q, Q)$ the greatest common divisor of~$q$ and~$Q$ in~$\ZZ[X]$.
\item We will say that \emph{$f$ is well-defined~$\mod{q}$} if~$(q, Q)=1$.
\item We define a subset of the primes by
$$ \cQ_f = \{p :\ f\text{ reduces to a quadratic polynomial modulo } p\}. $$
\item We will say that \emph{$f$ has total degree~$d$} if~$\deg P + \deg Q = d$.
\end{enumerate}
\end{definition}

Our main result is the following bound.

\begin{theorem}\label{th:bound-exp-automatic}
Let~$(a_n)$ be an automatic sequence, $f\in\QQ[X]$ be a rational function of total degree at most~$d\geq 1$, and~$q\geq 1$ such that~$f$ is well-defined~$\mod{q}$. Let~$q_1$ be the largest squarefree divisor of~$q$, coprime to~$k$ and having no prime factor in~$\cQ_f$:
$$ q_1 := \prod_{\substack{p\|q:\\p\not\in \cQ_f, p\nmid k}} p. $$
Then there exists~$c>0$, depending at most on~$d$ and the underlying automaton~$\cA$, such that
\begin{equation}\label{eq:bound-auto-nonquad}
\ssum{n\in \cI} a_n \e_q(f(n)) \ll_{\ee, \cA, d} |\cI|^{1+\ee} \Big(\frac1{q_1}+\frac{q^2}{q_1|\cI|^2}\Big)^c,
\end{equation}
for any interval of integers~$\cI \neq \varnothing$.

If~$f(X)$ is polynomial of degree exactly~$2$ and leading coefficient~$u/v$ with $(v, q)=1$, then
\begin{equation}\label{eq:bound-auto-quad}
\ssum{n\in \cI} a_n \e_q(f(n)) \ll_{\ee, \cA, u, v} |\cI|^{1+\ee} \Big(\frac{1}{q} + \frac{q}{|\cI|^2}\Big)^c,
\end{equation}
where the implied constant may now also depend on~$u$ and~$v$. 
\end{theorem}

In particular, if~$q=p$ is prime, and~$f$ does not reduce to a linear function~$\mod{p}$, we obtain
$$ \ssum{y<n\leq y+x} a_n \e_p(f(n)) \ll_{\ee, \cA, f} x^{1+\ee} \Big(\frac1p + \frac{p}{x^2}\Big)^c. $$
As mentioned earlier, this bound is non-trivial in the whole P\'{o}lya-Vinogradov range~$|\cI|\geq q^{1/2+\ee}$.
As an example, if~$s_2(n)$ denotes the sum of digits of~$n$ in base~$2$, then for some~$c>0$,
$$ \ssum{y < n \leq y+x \\ s_2(n)\text{ is even}} \e\Big(\frac{\bar{n}}{q}\Big) \ll_{\ee} x^{1-c\ee} \qquad (y\geq 0, x\geq 1)$$
for~$q$ prime and~$x^\ee \leq q \leq x^{2-\ee}$.

\begin{remarks} \hfill
\begin{enumerate}[1-]
\item Note that the bound~\eqref{eq:bound-auto-nonquad} is trivial when~$f$ is a linear or constant polynomial. 
This is clear for constant~$f$, and when~$f(X)=X$ for instance, it is easy to see that the stated bound fails with~$\cI=[0,q/2]\cap\ZZ$ and~$a_n=1$ for all~$n$.

\item As mentioned, we will prove a general statement (Proposition~\ref{prop:weyl} below) showing that for a bounded sequence of coefficients~$(K(n))_{n\geq 1}$, we obtain a non-trivial bound for~$\sum_{n\in \cI} a_n K(n)$ as soon as we have non-trivial bounds on two-point correlations sums of the kind
$$ \ssum{n\in \cI \\n\equiv a\mod{q}} K(n+r)\bar{K(n)} $$
with some mild uniformity in~$q$ and~$r$. For instance, this offers the possibility to take~$K(n)$ to be a more general algebraic trace function~\cite{FKM1, FKM2}, or Fourier coefficients of a $GL_2$ holomorphic cusp form~\cite{Blomer}.

\item The case when the automatic sequence is sparse, in the sense that~$\sum_{n\in\cI} |a_n| = o(|\cI|)$ as~$|\cI|\to\infty$, is more delicate as, then, the ``trivial bound'' obtained from the triangle inequality is possibly smaller than the right-hand sides of our bounds~\eqref{eq:bound-auto-nonquad} and~\eqref{eq:bound-auto-quad}. Our bounds yield a non-trivial saving as long as~$\sum_{n\in\cI} |a_n| \gg |\cI|^{1-\eta}$ and~$\eta>0$ is small enough, in the range~$|\cI|^{O(\eta)} \leq q \leq |\cI|^{2-O(\eta)}$. For instance, our results apply for numbers with one missing digit in a large enough base~$k\geq k_0$. Obtaining a good estimate for the smallest such~$k_0$ is a challenging question, which we do not address here; see~\cite{Maynard} for recent progress on the corresponding question for primes.
\end{enumerate}
\end{remarks}

Bounds of the type of Theorem~\ref{th:bound-exp-automatic} can be used to answer additive problems, see~\cite{FM}. We illustrate this by the following statement, concerning solutions to congruence equations.
\begin{theorem}\label{th:bound-additive}
Let~$\cS \subset\NN$ be a set of integers with the property that~$(a_n) = (\ind_{n\in\cS})$ is an automatic sequence; such a set is called automatic set. There exists~$r\in\NN$ and~$\delta>0$, depending only on the automaton~$\cA$ underlying~$(a_n)$, such that the following holds. 
For all rational fractions~$f_1, \dotsc, f_r$ none of which is a linear or constant polynomial, all $m\in\ZZ$ and all prime~$q$, the number~$N_\cS((f_j), q)$ of solutions to the congruence equation
$$ f_1(n_1) + \dotsb + f_r(n_r) \equiv m \mod{q} $$
with each~$n_j\in \cS\cap[1, q]$, is asymptotically
\begin{equation}\label{eq:estim-additive}
N_\cS((f_j), q) = \frac{|\cS\cap [1, q]|^r}q \big\{ 1 + O(q^{-\delta})\big\}.
\end{equation}
The implied constant may depend on~$f_1, \dotsc, f_r$ and~$\cS$.
\end{theorem}
\begin{remark}
As we have already remarked, constant sequences are automatic, so the above does not hold in general for small values of~$r$. 
It is however an important aspect that~$r$ does not grow with~$q$, and does not depend on~$(f_j)$.
\end{remark}

\subsection*{Context and overview}

There has been many works on correlations of automatic sequences with other arithmetic objects. Some of this interest has stemmed from questions of diophantine approximations and normality of numbers constructed from automatic sequences. For instance, Mauduit~\cite{Mauduit} obtains non trivial bounds on sums of the kind
\begin{equation}
\ssum{n\leq x} a_n \e(\alpha n)\label{eq:linear-phase}
\end{equation}
for irrationnal $\alpha$. See the references in~\cite{Mauduit} for more on the history of this question.\footnote{We emphasize that in these works, the case when~$\sum_{n\leq x}|a_n|=o(x)$ is particularly important. As we have already remarked, we do not focus on sparse sequences in this work.}

The method presented here, however, is related to partial progress on Sarnak's conjecture~\cite{Sarnak}. 
For automatic sequences of the kind of~$(-1)^{s_2(n)}$ (where we recall that~$s_2(n)$ is the sum of base-2 digits of~$n$), 
Mauduit and Rivat~\cite{MR-G} point out a certain property (which was later called ``carry property''), 
and show how it can be exploited in conjunction with the differencing method of Weyl and van der Corput together with strong estimates for the $L^{1}$ norm of the discrete Fourier transform of this sequence,
to obtain Sarnak's conjecture for this case; 
they also apply this method to show orthogonality to $\Lambda(n)$ which gives a prime number theorem.
Their approach was further formalized and generalized in~\cite{MR-RS} (see also~\cite{Hanna}), but the estimates on the $L^{1}$ norm were replaced by a so called ``Fourier property'' ($L^\infty$-bounds on the discrete Fourier transform). 
Finally, the second author recently showed Sarnak's conjecture for automatic sequences~\cite{mullner}, 
generalizing in particular results for synchronizing automatic sequences~\cite{synchronizing} and invertible automatic sequences~\cite{drmota2014,fourauthors}. 

The present work shows that both Mauduit-Rivat's ``carry property'', and the second author's structure theorems for automatic sequences, can be successfully combined with van der Corput differencing when handling algebraic exponential sums. At the heart of the bounds~\eqref{eq:bound-auto-nonquad} and~\eqref{eq:bound-auto-quad} lies Weil's bounds on exponential sums~\cite{Weil}.

In Section~\ref{sec:weil-bounds}, we state the precise version of Weil's bounds which we will use, and in Section~\ref{sec:auxil-automata} we quote auxiliary results on automata, mainly from~\cite{synchronizing,invertible,mullner}. In Section~\ref{sec:weyl-differentiation}, we prove a general statement (Proposition~\ref{prop:weyl}) linking generic sums over automatic sequence into differentiated sums over intervals. In Section~\ref{sec:proof-connected}, we prove Theorem~\ref{th:bound-exp-automatic} in a particular case, and in Section~\ref{sec:proof-final} we deduce the general case.

\section{Weil bounds}\label{sec:weil-bounds}

We begin by recalling from~\cite{Polymath} the convention regarding~$\e_q(f(n))$. Write in reduced form~$f(X) = P(X)/Q(X)$, with~$P, Q \in \ZZ[X]$. Given a prime power~$p^\nu$ with~$Q\not\equiv 0\mod{p^\nu}$, reduce~$P/Q \equiv P_1/Q_1 \mod{p^\nu}$. For~$n\in\ZZ$, we define a function of the pair~$(f, n)$ by
$$ \e_{p^\nu}(f ; n) = \begin{cases} \e\Big(\frac{P_1(n)\bar{Q_1(n)}}{p^\nu}\Big), & (Q_1(n), p)=1 \\ 0 & \text{otherwise.} \end{cases} $$
We will denote this by the slight abuse of notation~$\e_{p^\nu}(f(n))$. We extend this definition to arbitrary moduli~$q\geq 1$ with~$(q, Q)=1$ by the Chinese remainder theorem,
\begin{equation} \e_q(f(n)) := \prod_{p^\nu\|q} \e_{p^\nu}\Big(\frac{f(n)}{q/p^\nu}\Big).\label{eq:def-eq} \end{equation}

The purpose of this section is to justify the following bound on particular exponential sums.
\begin{lemma}\label{le:rational_differences}
Let~$x, s\geq 1$, $y\geq 0$, $q\geq 2$, $(a, r)\in\ZZ^2$, and~$f\in\QQ[X]$ of total degree at most~$d$, which is well-defined~$\mod{q}$. Then we have
\begin{equation}\label{eq:bound-incomplete-nonquad}
\ssum{y<n \leq y+x \\ n\equiv a\mod{s}} \e_q(f(n+r)-f(n)) \ll_{\ee, d} q^\ee \Big(\prod_{\substack{p\|q, p\nmid rs \\ p\not\in \cQ_f}}p\Big)^{-\frac12} \Big(\frac{x}s + q\Big).
\end{equation}
If~$f(X)=\displaystyle{\frac uv X^2}$ with~$(uq, v)=1$, then
\begin{equation}
\ssum{y < n \leq y+x \\ n\equiv a\mod{s}} \e_q(f(n+r)-f(n)) \ll \min\Big\{\frac{x}{s}+1, \Big\|\frac{2u\bar{v}rs}q\Big\|^{-1}_{\RR/\ZZ}\Big\}.\label{eq:bound-incomplete-quadratic}
\end{equation}
\end{lemma}
The bound claimed in~\eqref{eq:bound-incomplete-nonquad} corresponds to square-root cancellation in the part of the modulus which is squarefree, has no factor in~$\cQ_f$ and is coprime to~$rs$. 
We have assumed squarefreeness because it usually suffices in applications, and greatly simplifies the argument; 
the contribution of higher powers of primes could be studied in specific cases by elementary arguments (see lemmas 12.2 and 12.3 of~\cite{IK}).

The proof of Lemma~\ref{le:rational_differences} is based on the following Weil bound, which is a slightly weaker form of~\cite[proposition~4.6]{Polymath}.
\begin{lemma}[Weil~\cite{Weil}, Proposition~4.6 of \cite{Polymath}]\label{lem:weil}
Let~$q\geq 1$ be squarefree, and~$f\in\QQ(X)$ of total degree~$\leq d$, which is well-defined~$\mod{q}$. Then
$$ \sum_{n\mod{q}} \e_q(f(n)) \ll_{\ee, d} q^{1/2+\ee} (q, f')^{1/2}. $$
\end{lemma}

To deal with the factor~$(q, f')$, we will use the following lemma.
\begin{lemma}\label{lem:gcd}
Let~$f\in\QQ(X)$, which is not a polynomial of degree~$\leq 2$. Let~$q\geq 2$ be squarefree and such that for all~$p|q$, $p\not\in \cQ_f$, we have~$Q\not\equiv 0\mod{p}$. Then
$$ (q, f'(X+r) - f'(X) + \ell) \ll (q, r) \prod_{\substack{p\mid q, p\nmid r \\ p\in \cQ_f}} p \qquad (r, \ell\in\ZZ). $$
\end{lemma}

\begin{proof}
Write~$f=P/Q$ in reduced form, with~$P, Q \in \ZZ[X]$. It will be sufficient to prove that~$(p, f'(X+r)-f'(X)+\ell)=1$ when~$p$ is large enough in terms of the degree of~$P$ and~$Q$,~$p\not\in \cQ_f$,~$p\nmid 2r$. Suppose otherwise. Then by Lemma 4.5(i) of~\cite{Polymath}, we have~$(p, f(X+r) - f(X) + \ell X - c)=p$ for some class~$c\mod{p}$. Adding to~$f$ an appropriate quadratic polynomial, we may suppose~$(p, f(X+r)-f(X))=p$. Write~$P/Q=P_1/Q_1 \mod{p}$ with~$P_1, Q_1$ coprime. Then we deduce
$$ P_1(a)Q_1(a+r) \equiv Q_1(a)P_1(a+r) \mod{p}. $$
By coprimality, for all~$a\mod{p}$, $Q_1(a)\equiv 0$ implies~$Q_1(a+r)\equiv 0$. Iterating yields~$Q_1(a)\equiv 0$ for all~$a\mod{p}$. If~$p$ is large enough in terms of~$\deg Q$, we would obtain~$Q_1\equiv 0$ which is a contradiction. We deduce~$Q_1(a)\neq 0\mod{p}$ for all~$a$, so that~$P_1(X)/Q_1(X)$ takes a constant value and has no poles. If~$p$ is large enough in terms of~$\deg P$ and~$\deg Q$, we conclude that~$P_1 / Q_1$ is a constant polynomial, which again contradicts the hypothesis~$p\not\in \cQ_f$.
\end{proof}

\begin{proof}[Proof of Lemma~\ref{le:rational_differences}]
The bound~\eqref{eq:bound-incomplete-quadratic} is the simple bound for a geometric sum, therefore, we focus on proving~\eqref{eq:bound-incomplete-nonquad}. Changing variables, the LHS is
$$ \sum_{(y-a)/s < m \leq (y+x-a)/s} \e_q(f(a+ms+r)-f(a+ms)). $$
We cover the summation interval by at most~$1+x/sq$ intervals of length~$q$, and detect the size conditions on~$m$ by additive characters. We obtain
$$ \sum_{(y-a)/s < m \leq (y+x-a)/s} \e_q(f(a+ms+r)-f(a+ms)) \ll \frac x{sq}|S_0(q)| + \sum_{1\leq |\ell| \leq q/2} \frac{|S_\ell(q)|}\ell, $$
where
$$ S_\ell(q) = \sum_{m\mod{q}} \e_q(f(a+ms+r)-f(a+ms) + \ell m). $$
Let~$q_1$ be the largest divisor of~$q$ which is squarefree, coprime with~$rs$ and~$q/q_1$, and has no prime factor in~$\cQ_f$:
$$ q_1 = \prod_{\substack{p \| q,\ p\nmid rs \\ p\not\in \cQ_f}} p. $$
By the Chinese remainder theorem, we may write~$S_\ell(q) = T_1 T_2$, where
$$ T_1 = \sum_{m\mod{q_1}} \e_{q_1}((q/q_1)^{-1}(f(a+ms+r)-f(a+ms) + \ell m)), $$
and
$$ T_2 = \sum_{m\mod{q/q_1}} \e_{q/q_1}(q_1^{-1}(f(a+ms+r)-f(a+ms) + \ell m)). $$
On~$T_2$ we use the trivial bound~$|T_2|\leq q/q_1$. Concerning~$T_1$, by Lemma~\ref{lem:weil} applied with~$f(X)$ replaced by~$f(a+sX+r)-f(a+sX)+\ell X$, we get
$$ T_1 \ll_\ee q_1^{\frac12+\ee} (q_1, \ell + sf'(a+sX+r)-sf'(a+sX))^{\frac12}. $$
Let~$v\in\ZZ$ be such that~$sv\equiv 1\mod{q}$. We apply Lemma~\ref{lem:gcd} with~$r\gets rv$ and~$f(X)\gets f(a+sX)$. We obtain~$(q_1, \ell + sf'(a+sX+r)-sf'(a+sX)) = O(1)$, therefore~$|T_1|=O(q_1^{\frac12+\ee})$, and so
$$ |S_\ell(q)| \ll q^{1+\ee}q_1^{-\frac12}. $$
This leads to the desired conclusion.
\end{proof}

\section{Auxiliary results on automata}\label{sec:auxil-automata}

We quote in this section a few results from the literature which we will use in our proof of Theorem~\ref{th:bound-exp-automatic}. 

From now on, $(a_n)$ denotes a fixed automatic sequence corresponding to a stongly connected automaton $\cA = (Q',\Sigma, \delta',q'_0,Q_0)$, where $\delta'(q'_0,0) = q'_0$. 
We follow some ideas and the notion of \cite{mullner} and consider a naturally induced transducer 
$\mathcal{T}_{\cA} = (Q,\Sigma, \delta,q_0,\Delta, \lambda)$, where $Q \subset (Q')^{n_0}, \pi_1(q_0) = q'_0$, $\delta$ a transition function which is synchronizing\footnote{This means that there exists a synchronizing word $\mathbf{w}_0$, i.e., $\delta(q_0,\mathbf{w}_0) = \delta(q,\mathbf{w}_0)$ for all $q \in Q$.}
and an output function
$\lambda: Q\times \Sigma \to \Delta \subset S_{n_0}$ which ``attaches'' to each transition in the naturally induced transducer a permutation.

A transducer can be viewed as a mean to define functions: on input $\mathbf{w} = w_1w_2\ldots w_r$ the transducer enters states 
$q_0 = \delta(q_0, \ee), \delta(q_0,w_1),\ldots,\delta(q_0,w_1w_2\ldots w_r)$ and produces the outputs
\begin{align*}
  \lambda(q_0,w_1),\lambda(\delta(q_0,w_1),w_2),\ldots,\lambda(\delta(q_0,w_1w_2\ldots w_{r-1}),w_r).
\end{align*}
The function $T(\mathbf{w})$ is then defined as
\begin{align*}
  T(\mathbf{w}) := \prod_{j = 0}^{r-1} \lambda(\delta(q_0,w_1w_2\ldots w_{j}),w_{j+1}).
\end{align*}
We also define the slightly more general form,
\begin{align*}
  T(q,\mathbf{w}) := \prod_{j = 0}^{r-1} \lambda(\delta(q,w_1w_2\ldots w_{j}),w_{j+1}).
\end{align*}

Proposition 2.5 of \cite{mullner} shows how the original automaton and the naturally induced transducer are related, namely
\begin{align}\label{eq:automaton_transducer}
  a_n = \tau(\delta'(q'_0,(n)_k)) = \tau(\pi_1 (T(q_0,(n)_k) \cdot \delta(q_0,(n)_k))).
\end{align}

The following theorem highlights an important closure property of naturally induced transducers.
\begin{theorem}[Theorem 2.7 of~\cite{mullner}]\label{th:full_G}
  Let $\cA$ be a strongly connected automaton. There exists a minimal $d\in \NN$, $m_0 \in \NN$,
  a naturally induced transducer $\mathcal{T}_{\cA}$
  and a subgroup $G$ of $\Delta$ such that the following two conditions hold.
  \begin{itemize}
    \item For all $q \in Q, \mathbf{w} \in (\Sigma^{d})^{*}$ we have $T(q,\mathbf{w})\in G$.
    \item For all $g \in G, q,\overline{q}\in Q$ and $n \geq m_0$ it holds that
      \begin{align*}
	\{T(q,\mathbf{w}): \mathbf{w} \in \Sigma^{nd}, \delta(q,\mathbf{w}) = \overline{q}\} = G.
      \end{align*}
  \end{itemize}
  $d$ and $m_0$ only depend on $\cA$, but not on its initial state $q'_0$.
\end{theorem}

Finally, \cite[Corollary 2.26]{mullner} shows that there exists $\cA = (Q',\Sigma,q'_0,\delta',\tau)$ generating $(a_n)$ such that $d(\cA) = 1$ and we consider a 
naturally induced transducer which fulfills Theorem~\ref{th:full_G}.

One crucial idea in \cite{mullner} was that the functions $T$ and $\delta$ corresponding to a naturally induced transducer behave ``independently'' of each other.
Thus, we start by giving an important property of synchronizing automata.
\begin{lemma} \label{le:synchronizing}
Let $\mathcal{\cA}$ be a synchronizing DFAO with synchronizing word $\mathbf{w} \in \Sigma^{m_0}$. 
There exists $\eta > 0$ depending only on $m_0$ and $k$ such that the number of integers~$n\in(y, y+x]$ such that
\begin{align*}
  \delta(q,(n)_k) \neq \delta(q,(n)_k^{\lambda})
\end{align*}
is bounded by $O(x k^{-\eta \lambda})$ uniformly for $\lambda < \floor{\log_k(x)}$ and~$y\geq 0$. 
Here,~$(n)_k^{\lambda}$ denotes the digital representation of~$n$ truncated at the~$k$-th digit, 
in other word~$(n)_k^{\lambda} = (m)_k$ where~$m\in[0, k^\lambda)\cap\NN$ and~$m\equiv n\mod{k^\lambda}$.
\end{lemma}
\begin{proof}
  See Lemma~2.2 of~\cite{synchronizing}.
\end{proof}

The next result is the carry property for automatic sequences, or more precisely $T$.

\begin{definition}\label{def:1}
  A function $f:\NN\rightarrow U_d$ has the carry property if there exists $\eta >0$ such that uniformly for $\lambda,\alpha, \rho\in \NN$
  with $\rho<\lambda$, the number of integers $0\leq \ell <k^{\lambda}$ such that there exists $(n_1,n_2) \in \{0,\ldots,k^{\alpha}-1\}^2$ with
  \begin{align}\label{eq:carry_violation}
    f(\ell k^{\alpha} + n_1 + n_2)^{H} f(\ell k^{\alpha} + n_1) \neq f_{\alpha + \rho}(\ell k^{\alpha} + n_1 + n_2)^{H} f_{\alpha + \rho}(\ell k^{\alpha} + n_1)
  \end{align}
  is at most $O(k^{\lambda-\eta \rho})$ where the implied constant may depend only on $k$ and $f$.
\end{definition}
\begin{lemma}[Lemma~4.9 of~\cite{mullner}]\label{le:def_1}
  Definition~\ref{def:1} holds -- uniformly in $r$ -- for $f (n) = D(T(n+r))$ where $D$ is a unitary and irreducible representation of $G$, $\eta$ is given by \cite[Lemma 2.2]{synchronizing}
  and the implied constant does not depend on $r$.
\end{lemma}

To use the carry property efficiently, we need the following lemma which is a generalization of Van-der-Corput's inequality.
\begin{lemma}\label{lemma:van-der-corput}
  Let $y\geq 0$, $x\geq 1$ and $Z(n)\in\CC^{d\times d}$ be given for all $n\in(y, y+x]$.
  Then we have for any real number $R\geq 1$ and any integer $k\geq 1$ the estimate
  \begin{equation}\label{eq:van-der-corput}
    \norm{\sum_{y<n\leq y+x} Z(n)}_{F}^2
    \le
    \frac{x+k(R-1)+1}{R}
    \sum_{\abs{r}<R}\rb{1-\frac{\abs{r}}{R}}
    \ssum{y<n, n+kr\leq y+x}\mathrm{tr} \rb{ Z(n+kr)^{H} Z(n) }
  \end{equation}
  where $\mathrm{tr}(Z)$ denotes the trace of $Z$, and~$\|Z\|_F$ the Frobenius norm of~$Z$.
\end{lemma}
\begin{proof}
See Lemma 5 of~\cite{invertible}.
\end{proof}

We now quote results from representation theory. Consider a finite groupe~$G$. 
A representation $D$ is a continuous homomorphism $D: G \rightarrow U_d$, where $U_d$ denotes the group of unitary $d\times d$ matrices
over $\CC$. A representation $D$ is called irreducible if there exists no non-trivial subspace $V \subset U_d$ such that $D(g)V \subseteq V$ holds for all $g\in G$.
It is well-known that there only exist finitely many equivalence classes of unitary and irreducible representations of $G$ 
(see for example \cite[Part I, Section 2.5]{rep_finite}). The Peter-Weyl Theorem (see for example \cite[Chapter 4, Theorem 1.2]{representations}) 
states that the entry functions of irreducible representations (suitable renormalized) form a orthonormal basis of $L^2(G)$.
Thus we can express any function $f:G\to \CC$ by these $M_0$ entry functions:
  
\begin{lemma}\label{le:reps}
  Let $G$ be a finite group. There exists $M_0 \in \NN$ and $M_0$ irreducible unitary representations $(D^{(\ell)})_{0\leq \ell < M_0}$ of $G$, not necessarily distinct and written as matrices~$D^{(\ell)} = (d_{i,j}^{(\ell)})_{i, j}$, such that for any $f: G\to \CC$ there exist coefficients $(c_{\ell})$ and indices~$(i_\ell)$, $(j_\ell)$ with
  \begin{align*}
  f(g) = \sum_{0\leq \ell < M_0} c_{\ell} d_{i_{\ell}j_{\ell}}^{(\ell)}(g)
\end{align*}
for all $g\in G$ and $\sum \abs{c_{\ell}} \ll \norm{f}_1$.
\end{lemma}

\section{Van der Corput differentiation}\label{sec:weyl-differentiation}

The following proposition reduces the study of automatic sequences with strongly connected underlying automaton, to bounds on correlations sums.
\begin{proposition}\label{prop:weyl}
Let~$g:\NN_{>0} \to \CC$ be a function with~$|g(n)|\leq 1$,~$x\geq 1$, $y\geq 0$ be real numbers, and~$(a_n)$ be an automatic sequence in base~$k$, with strongly connected underlying automaton~$\cA$. Denote
\begin{equation}\label{eq:def-shifted-sum}
U(x, y ; h ; q, a) := \ssum{y<n\leq y+x \\ n\equiv a\mod{q}} g(n)\bar{g(n+h)}.
\end{equation}
Then for some~$\eta>0$ depending on~$\cA$, all~$\lambda_1, \lambda_2\in\NN$ with~$M := k^{\lambda_1}$, $R := k^{\lambda_2}$ satisfying~$RM^2 \leq x/10$, we have
$$ \Big|\ssum{y<n \leq y+x} a_n g(n) \Big| \ll xM^{-\eta} + \sum_{0\leq m < M}\Big(\frac{x}{RM} \sum_{0\leq r < R} \ssum{0\leq m' < RM^2 \\ m' \equiv m\mod{M}} |U(x, y ; rM ; RM^2, m')|\Big)^{1/2}. $$
\end{proposition}
\begin{remark}
It was proved by Sarnak that his Möbius randomness conjecture for all deterministic flows would follow from the Chowla conjecture (see~\cite{Sarnak,Tao}) concerning correlations of the Möbius function. 
Proposition~\ref{prop:weyl} could be interpreted as a quantified version of this phenomenon for automatic sequences; 
we see that in this case, binary correlations provide sufficient information. 
Note however that the moduli~$q=RM^2$ of the arithmetic progressions involved in our statement are rather large compared with the shifts $h=rM$.
\end{remark}

We prove Proposition~\ref{prop:weyl} in the remainder of this section.

\subsection{Naturally induced transducer}

We use the concept of naturally induced transducer to rewrite the sequence~$a_n$.
We (still) consider a naturally induced transducer which fulfills Theorem~\ref{th:full_G}.
By \eqref{eq:automaton_transducer}, we can rewrite~$a_n=\tau(\pi_1 (T(q_0,(n)_k) \cdot \delta(q_0,(n)_k)))$. Therefore,
\begin{align*}
  a_n = \sum_{q \in Q} \sum_{\sigma \in G} \tau(\pi_1(\sigma \cdot q)) \ind_{[T(q_0,(n)_k) = \sigma]} \ind_{[\delta(q_0,(n)_k) = q]}.
\end{align*}
Let
$$ \cI = \ZZ \cap (y, y+x] $$
with~$y\in\NN_{\geq 0}$ and~$x\in \NN_{>0}$. The above allows us to rewrite
\begin{align}\label{eq:S_0toS_1}
  S_0(\cI) := \ssum{y<n\leq y+x} a_n f(n) = \sum_{q \in Q} \sum_{\sigma \in G} \tau(\pi_1(\sigma \cdot q)) S_1(\cI; \sigma, q)
\end{align}
where 
\begin{align*}
  S_1(\cI; \sigma, q) := \sum_{y<n\leq y+x} \ind_{[T(q_0,(n)_k) = \sigma]} \ind_{[\delta(q_0,(n)_k) = q]} f(n).
\end{align*}

This implies
\begin{align*}
  \abs{S_0(\cI)} \ll_\cA \sum_{q \in Q} \sum_{\sigma \in G} \abs{S_1(\cI; \sigma, q)}.
\end{align*}

\subsection{Van der Corput differencing and the carry property}

Let~$1\leq M \leq x$, $M = k^{\lambda_1}$ be a power of~$k$, to be determined later.
We use the fact that it is usually sufficient to read the last few digits of $(n)_k$ to determine $\delta(q,(n)_k)$, see Lemma~\ref{le:synchronizing}.
This allows us to rewrite
\begin{equation}\label{eq:S_1toS_2}
\begin{aligned}
S_1(\cI; \sigma, q) {}& = \sum_{y<n\leq y+x} \ind_{[T(q_0,(n)_k) = \sigma]} \ind_{[\delta(q_0,(n)_k) = q]} g(n)\\
  &{} = \sum_{0\leq m < M} \ind_{[\delta(q_0,(m)_k) = q]} S_2(\cI ; m, \sigma) + O(x M^{-\eta}),
\end{aligned}
\end{equation}
where $\eta>0$ only depends on the length of the synchronizing word of the naturally induced transducer $\mathcal{T}_{A}, \mathbf{w}_0$, and 
\begin{align*}
  S_2(\cI ;m, \sigma) := \ssum{y<n\leq y+x\\n\equiv m \bmod M} \ind_{[T(q_0,(n)_k) = \sigma]} g(n).
\end{align*}
We use ideas of representation theory to deal with $\ind_{[T(q_0,(n)_k) = \sigma]}$.
  By Lemma~\ref{le:reps}, we can write
  \begin{align*}
    \ind_{[T(q_0,(n)_k) = \sigma]} = \sum_{0\leq \ell < M_0} c_{\ell} d_{i_{\ell}j_{\ell}}^{(m_\ell)}(T(q_0,(n)_k)),
  \end{align*}
  for some unitary and irreducible representations $D^{(m')}$.
  This gives
  \begin{align*}
    \ssum{y<n\leq y+x\\n\equiv m \bmod M} \ind_{[T(q_0,(n)_k) = \sigma]} g(n) = 
	\sum_{0\leq \ell < M_0} c_{\ell}\ssum{y<n\leq y+x\\n\equiv m \bmod M} d_{i_{\ell}j_{\ell}}^{(m_\ell)}(T(q_0,(n)_k)) g(n)
  \end{align*}
  and
  \begin{align*}
    \Bigg|\ssum{y<n\leq y+x\\n\equiv m \bmod M} d_{i_{\ell}j_{\ell}}^{(m_\ell)}(T(q_0,(n)_k)) g(n)\Bigg| \leq 
	\Bigg\|\ssum{y<n\leq y+x\\n\equiv m \bmod M} D^{(m_\ell)}(T(q_0,(n)_k)) g(n) \Bigg\|_{F},
  \end{align*}
  where $\|.\|_F$ denotes the Frobenius norm.
  
  Thus, we find
  \begin{align*}
    \abs{S_2(\cI;m, \sigma)} \leq \sum_{0\leq \ell < M_0} \abs{c_{\ell}} \norm{S_3(\cI;m, D^{(m_{\ell})})}_{F},
  \end{align*}
  where
  \begin{align*}
    S_3(\cI;m, D) := \ssum{y<n\leq y+x\\n\equiv m \bmod M} D(T(q_0,(n)_k)) g(n).
  \end{align*}

  This gives in total
  \begin{align}\label{eq:S_0toS_3}
    \abs{S_0(\cI)} \ll_{\cA} \max_{D} \sum_{0\leq m<M} \norm{S_3(\cI;m,D)}_F + O(x M^{-\eta}).
  \end{align}

We use Lemma~\ref{lemma:van-der-corput} for the sequence $Z(n) = D(T(q_0,(nM+m)_k)) g(nM+m)$:
\begin{align*}
  &\norm{S_3(\cI;m,D)}_{F}^{2} \leq \frac{x M^{-1} + M(R-1)+1}{R} \sum_{\abs{r}<R}\rb{1-\frac{\abs{r}}{R}} \mathrm{tr} \rb{S_4(\cJ_r; m, D, r)},
\end{align*}
where $\cJ_r := \{n: y<n, n+r M \leq y+x\}$ and 
\begin{align*}
  S_4(\cJ, m, D, r) := & 
    \ssum{n \in \cJ\\ n \equiv m \bmod M}\rb{  D(T(q_0,(n+ rM)_k))^{H}D(T(q_0,(n)_k))}\\
    &\qquad \qquad g(n)\bar{g(n+rM)}.
\end{align*}
We choose $R = k^{\lambda_2}$ and $\lambda_2\in \NN$ subject to~$RM^2 < x/10$, which gives
\begin{align}\label{eq:S_3toS_4}
  &\norm{S_3(\cI;m,D)}_{F}^{2} \ll \frac{x}{RM} \sum_{0\leq r < R}\norm{S_4(\cI; m, D, r)}_F + O(Rx/M),
\end{align}
where the error term is due to the replacement of $\cJ_r$ by $\cI$.

Letting temporarily~$a = y+1$, we rewrite $n = n_1 RM + n_2M + m + a$ to find
\begin{align*}
  &D(T(q_0,(n+ rM)_k))^{H} D(T(q_0,(n)_k))\\
  &\qquad = D(T(q_0, ( n_1 RM + (n_2 M + m + a) + r M)_k))^{H} D(T(q_0,(n_1 RM +(n_2 M + m + a))_k)).
\end{align*}
We apply Lemma~\ref{le:def_1} with $\alpha = \lambda_1 + \lambda_2, \rho = \lambda_1$ and $\ell = n_1$.
This gives
\begin{align*}
  &D(T(q_0,(n+ rM)_k))^{H} D(T(q_0,(n)_k))\\
  &\qquad = D(T_{2\lambda_1 + \lambda_2}(q_0, (n_1 RM + (n_2M + m + a) + r M)_k))^{H} D(T_{2\lambda_1+\lambda_2}(q_0,(n_1 RM +(n_2 M+m+a))_k))\\
  &\qquad = D(T_{2\lambda_1 + \lambda_2}(q_0, (n+ r M)))^{H} D(T_{2\lambda_1 + \lambda_2}(q_0,(n)_k)),
\end{align*}
for all but $O(xR^{-1}M^{-1-\eta})$ values of $n_1 \in [0, x /RM)$ and, 
therefore, for all but $O(x M^{-1-\eta})$ values of $n\in\cI$ (for fixed $m$).

Thus, we find
\begin{align}\label{eq:S_4toS_5}
\begin{split}
  S_4(\cI;m,D,r) &= \ssum{0\leq m' < RM^2\\m' \equiv m \bmod M} D(T_{2\lambda_1+\lambda_2}(q_0, m'+ rM))^{H} D(T_{2\lambda_1 + \lambda_2}(q_0,m')) S_5(\cI;m',r)\\
     &\qquad + O(x M^{-1-\eta}),
\end{split}
\end{align}
where
\begin{equation}\label{eq:post-corput}
  S_5(\cI;m',r) := \ssum{y<n\leq y+x\\n \equiv m' \bmod RM^2} g(n)\bar{g(n+rM)}.
\end{equation}
Note that the trivial estimate $S_5 = O(x/(RM^2))$ gives back the trivial estimate $S_0 \ll x$, so non-trivial bound on~$S_5$ gives a non-trivial bound on~$S_0$.

Combining \eqref{eq:S_3toS_4} and \eqref{eq:S_4toS_5} gives
\begin{align}\label{eq:S_3toS_5}
  \norm{S_3(\cI;m,D)}_{F}^{2} &\ll \frac {x}{RM} \sum_{0\leq r < R }\norm{S_4(\cI; m, D, r)}_F + O(x)\\
      &\ll_{\cA} \frac{x}{RM} \sum_{0\leq r< R} \ssum{0\leq m' < RM^2\\m' \equiv m \bmod M} \abs{S_5(\cI;m',r)} + O(x^2 M^{-2-\eta}).
\end{align}
This together with the definition~\eqref{eq:def-shifted-sum} finishes the proof of Proposition~\ref{prop:weyl}.

\section{Proof of Theorem~\ref{th:bound-exp-automatic} in the strongly connected case}\label{sec:proof-connected}

From Proposition~\ref{prop:weyl}, we will deduce Theorem~\ref{th:bound-exp-automatic} in the following special case.
\begin{proposition}\label{prop:exp-automatic-connected}
Theorem~\ref{th:bound-exp-automatic} holds for sequences~$(a_n)$ whose underlying automata are strongly connected.
\end{proposition}
The proof is split in two cases, according to whether or not the rational fraction~$f$ is a quadratic polynomial.

\subsection{The non-quadratic case}

We assume first that~$f$ is not a quadratic polynomial. Let~$R = M = k^\lambda$, and
$$ q_1 = \prod_{\substack{p\|q, p\nmid k\\ p\not\in \cQ_f}} p. $$
We assume also that~$x\geq q q_1^{-\frac12}$ without loss of generality, since otherwise the right-hand side of~\eqref{eq:bound-auto-nonquad} is larger than the trivial bound~$O(x)$ 
for the left-hand side. Recall the definition~\eqref{eq:def-shifted-sum}. We use Lemma~\ref{le:rational_differences} with~$g(n) = \e_q(f(n))$ and our choice of~$R$ and $M$, to find the following estimate
\begin{align*}
 \sum_{0\leq r<k^{\lambda}} \ssum{0\leq m' < k^{3\lambda}\\m' \equiv m \bmod k^{\lambda}} \abs{U(x, y ; rM ; RM^2, m')}
	    \ll_{\ee} {}&  \sum_{0\leq r < k^{\lambda}} k^{2\lambda} q^{\ee} \rb{\frac{x}{k^{3\lambda}} + q}  
 		  \Bigg(\prod_{\substack{p | q, p\nmid rk\\p \notin T_f}} p^{-1/2}\Bigg)\\
	  \ll {}& k^{2\lambda}\rb{\frac{x}{k^{3\lambda}} + q} q_1^{-1/2} q^{\ee} 
		  \Bigg(\sum_{0\leq r<k^{\lambda}} (r, q)^{1/2} \Bigg).
\end{align*}
Thus, we find by Proposition~\ref{prop:weyl} that for some~$\eta>0$ depending on~$\cA$,
\begin{align*}
  \abs{\ssum{n\in \cI} a_n \e_q(f(n))} &\ll_{\cA,\ee} x \rb{\frac{k^{\lambda}}{q_1^{1/2}} + \frac{q k^{4\lambda}}{x q_1^{1/2}}}^{1/2} q^{\ee/2} 
		  + O(x k^{-\lambda \eta/2})
\end{align*}
uniformly in $\lambda$ such that~$k^{3\lambda} < x/10$. We choose $\lambda$ such that~$k^{\lambda} \asymp_k \min((q_1^{\frac12}q^{-1}x)^{1/8},(q_1)^{1/4})$. 
This gives
\begin{equation}
\abs{S_0(\cI)} \ll_{\cA, k, \ee, d}  x q^\ee \Bigg(\frac1{q_1^{1/4}} + \Bigg(\frac{q}{q_1^{1/2}x}\Bigg)^{1/8}\Bigg)^{1/2} + x \Bigg(\frac{1}{q_1^{1/4}} + \Bigg(\frac{q}{q_1^{1/2}x}\Bigg)^{1/8}\Bigg)^{\eta},
\end{equation}
and implies our claimed bound~\eqref{eq:bound-auto-nonquad} for $c = \min(\eta/16,1/32)$.

\subsection{The quadratic case}

Here again we assume that~$g(n) = \e_q(f(n))$. If~$f$ is quadratic, then for the purpose of bounding~\eqref{eq:def-shifted-sum} we may assume~$f(X) = \frac uv X^2$ with~$v\neq 0$ and $(qu, v)=1$. Let~$s=RM^2$, where~$R$ and~$M$ are powers of~$k$ satisfying~$1\leq RM < x/10$. By Lemma~\ref{le:rational_differences}, we have
$$ |U(x, y ; rM ; RM^2, m')| \ll \min\Big( \frac {x}{RM^2}, \Big\| \frac{2u\bar{v} rRM^2}{q}\Big\|^{-1}\Big). $$
Assume now that~$(RM)^2 < q/(4u)$, which does not contradict the hypotheses~$R, M\geq 1$ if we let~$q$ be large enough in terms of~$u$. Then
$$ \frac{2u\bar{v} rRM^2}{q}  \equiv  \frac{2urRM^2}{qv} - \frac{2u\bar{q} r RM^2}{v} \quad \mod{1}. $$
By our hypothesis~$(RM)^2 < q/(4|u|)$, as soon as~$v\nmid 2 r RM^2$, we obtain
$$ \Big\|  \frac{2u\bar{v} rRM^2}{q} \Big\| \geq \frac1v - \Big|\frac{2urRM^2}{qv}\Big| \geq \frac1{2v} \gg_f 1. $$
If, on the other hand, $v \mid 2 r RM^2$, we obtain
$$ \Big\|  \frac{2u\bar{v} rRM^2}{q} \Big\| = \Big| \frac{2urRM^2}{qv} \Big| = \frac{2|u|rRM^2}{q|v|} $$
since the latter is less than~$1/2$, again by our hypothesis~$(RM)^2 < q/(4|u|)$. In any case, we obtain
$$ |U(x, y ; rM ; RM^2, m')| \ll \min\Big( \frac{x}{RM^2}, \frac{q}{rRM^2}\Big), $$
and therefore
$$ \ssum{0\leq m' < RM^2 \\m' \equiv m \bmod M} |U(x, y ; rM ; RM^2, m')| \ll \min\Big( \frac{x}{M}, \frac{q}{rM}\Big). $$
We sum the previous bound over~$r<R$. We obtain
\begin{align*}
\sum_{r<R} \ssum{0\leq m' < RM^2\\m' \equiv m \bmod M} \abs{U(x, y ; rM ; RM^2, m')}
{}& \ll \Big(\frac{x}M + \sum_{1\leq r < R}  \frac{q}{rM}\Big) \\
{}& \ll \frac{x + q\log R}{M}.
\end{align*}
We now pick~$R \asymp_k \min\{x/M^2, \sqrt{q}/M\}$, and find by Proposition~\ref{prop:weyl}
\begin{align*}
\abs{\sum_{n\in \cI} a_n \e_q(f(n)) } \ll_{\cA, k,\ee,d} (xq)^\ee \Big(x(x+q)\Big(\frac{M^2}q + \frac{M}{q^{1/2}}\Big)\Big)^{1/2} + x M^{-\eta/2}.
\end{align*}
If~$x\leq q$, we pick~$ M \asymp_k (x/\sqrt{q})^{1/(1+2\eta)}$, and if~$x>q$, we pick~$M\asymp_k q^{1/(2+4\eta)}$. We find in any case
$$ \abs{\sum_{n\in \cI} a_n \e_q(f(n))} \ll x^{1+\ee}\Big(\frac1q + \frac{q}{x^2}\Big)^c $$
with~$c = \eta / (2 + 4\eta)$, and our claimed bound follows.

\section{Proof for non strongly connected automata}\label{sec:proof-final}

We will deduce the full generality of Theorem~\ref{th:bound-exp-automatic} from Proposition~\ref{prop:exp-automatic-connected} and the following fact.
\begin{proposition}\label{prop:to-connected}
Let~$g:\NN\to\CC$ be a function with~$|g(n)|\leq 1$, and assume that for every strongly connected automatic sequence~$\bb = (b_n)$, there is a non-decreasing function~$E(\bb, \cdot) : \RR_+ \to \RR_+$ that
\begin{equation}
\Big| \ssum{y < n \leq y+x} b_n g(n+r) \Big| \leq E(\bb, x) \qquad (r\in\ZZ, y\geq 0, x\geq 1).\label{eq:prop-connect-hypo}
\end{equation}
Then for any automatic sequence~$(a_n)$, not necessarily strongly connected, we may associate a finite set~$\{\bb^{(j)} = (b^{(j)}_n)\}_{j=1}^J$ of strongly connected automatic sequences, and a positive number~$\delta>0$ such that for all~$y\geq 0$, $x\geq 1$ and~$\sigma\in\NN$ with~$K := k^\sigma \in [1, x]$, we have
\begin{equation}
\Big| \ssum{y < n \leq y+x} a_n g(n) \Big| \ll \big\{x^{1-\delta} K^\delta + x K^{-1} \max_j E(\bb^{(j)},K) \big\}.\label{eq:bound-auto-general}
\end{equation}
\end{proposition}
\begin{remark}
It is important to note the requirement that the hypothesized upper-bound~\eqref{eq:prop-connect-hypo} is uniform with respect to~$r$.
\end{remark}
\begin{proof}[Proof of Proposition~\ref{prop:to-connected}]
Let~$\cA = (Q, \Sigma, \delta, q_0, \tau)$ be the automaton underlying~$(a_n)$, and define
$$ {\mathcal R} := \{r\in \NN : \delta(q,(r)_k) \text{ belongs to a final component of } \cA \text{ for any } q\in Q\}. $$
Then we have the uniform bound
\begin{equation}
|(y, y+x] \cap \NN \smallsetminus {\mathcal R}| \ll x^{1-\delta} \qquad (y\geq 0, x\geq 1)\label{eq:bound-connect-B}
\end{equation}
for some~$\delta>0$ depending on~$\cA$. We let~$\{(b^{(j)}_n)\}_{j=1}^J$ be the finite set of all automatic sequences associated with final components of~$\cA$ (as described in Proposition~2.25 of~\cite{mullner}), with the same output function~$\tau$.

We consider some fixed~$y\geq 0$ and~$x\geq 1$. For the purpose of proving the bound~\eqref{eq:bound-auto-general}, we may assume that~$x$ is large enough in terms of~$\cA$. Let~$\sigma\in\NN$ with~$1\leq K := k^\sigma \leq x$. 
We split the sum on the left-hand side of~\eqref{eq:bound-auto-general} in congruence classes~$\mod{K}$, getting
$$ \ssum{y \leq n < y+x} a_n g(n) = \sum_{r\geq 0} \ssum{0\leq n < K \\ y \leq r K + n < y+x} a_{rK+n} g(r K + n). $$
Note that the sum over~$n$ is void unless~$r\in (y/K-1, (y+x)/K)$. From this fact, the bound~\eqref{eq:bound-connect-B} and our hypothesis~$\|g\|_\infty \leq 1$, we obtain
$$ \sum_{r\geq 0} \ssum{0\leq n < K \\ y \leq r K + n < y+x} a_{rK+n} g(r K + n) = \ssum{r\geq 0 \\ r\in \cR} \ssum{0\leq n < K \\ y \leq r K + n < y+x} a_{rK+n} g(r K + n) + O(x^{1-\delta} K^\delta). $$
For~$r\in \cR$, our automaton reads numbers from left to right, so that~$a_{rK+n} = b_n^{(j)}$ for some~$j$ (depending on~$r$); we recall that there are only finitely many possibilities for~$j$. Therefore,
$$ \Bigg| \ssum{r\geq 0 \\ r\in \cR} \ssum{0\leq n < K \\ y \leq r K + n < y+x} a_{rK+n}  g(r K + n) \Bigg| \leq \sum_{\frac yK-1 < r < \frac{y+x}K} \max_j \Bigg| \ssum{0\leq n < K \\ y\leq rK + n < y+x} b_n^{(j)} g(rK+n)\Bigg|. $$
In the inner sum, the size conditions on~$n$ describe an interval of length~$K$, for all but at most two values of $r$. Gathering the above and using our hypothesis~\eqref{eq:prop-connect-hypo}, we find
\begin{align*}
\Bigg| \ssum{y \leq n < y+x} a_n g(n) \Bigg| {}&\ll \sum_{\frac yK-1 \leq r < \frac{y+x}K} \max_j E(\bb^{(j)}, K)  + x^{1-\delta} K^\delta \\
{}& \ll x K^{-1} \max_j E(\bb^{(j)}, K) + x^{1-\delta} K^\delta
\end{align*}
as claimed.
\end{proof}

\begin{proof}[Proof of Theorem~\ref{th:bound-exp-automatic}]
We consider the case of~\eqref{eq:bound-auto-nonquad}; the argument~\eqref{eq:bound-auto-quad} is similar and slightly simpler. 
Proposition~\ref{prop:exp-automatic-connected} shows that the estimate~\eqref{eq:bound-auto-nonquad} holds when the sequence~$(a_n)$ is associated to a strongly connected automaton. Note that the upper-bound~\eqref{eq:bound-auto-nonquad} depends only on the total degree of~$f$ (while~\eqref{eq:bound-auto-quad} depends only on the leading coefficient of~$f$). Moreover, if~$r\in\ZZ$ and~${\tilde f}(X) := f(X+r)$, then~$\cQ_{\tilde f} = \cQ_f$. We deduce that the bounds~\eqref{eq:bound-auto-nonquad} holds also, with the same implied constant, for the quantity
$$ \sum_{n\in \cI} a_n \e_q(f(n + r))  $$
uniformly in~$r\in\ZZ$, when the automaton underlying~$(a_n)$ is strongly connected. The hypothesis~\eqref{eq:prop-connect-hypo} is therefore satisfied with
$$ E((a_n), x) = B_{\cA,\ee} x^{1+\ee} \Big(\frac 1{q_1} + \frac{q^2}{q_1 x^2}\Big)^{c}, $$
where~$c>0$ depends on~$\cA$ and~$B_{\cA, \ee}$ depends at most on~$\cA$ and~$\ee$.

Assume now that~$(a_n)$ is not associated to a strongly connected automaton. For all~$K\in[1, x]$ which is a power of~$x$, we obtain by Proposition~\ref{prop:to-connected} the bound
$$ \Bigg|\ssum{y<n\leq y+x} a_n \e_q(f(n))\Bigg| \ll x^{1-\delta} K^\delta + xK^{\ee} \Big(\frac 1{q_1} + \frac{q^2}{q_1 K^2}\Big)^{c_1} $$
for some~$c_1>0$ depending on~$\cA$. We optimize by letting~$K \asymp_k \min\{x, (xq^2q_1^{-1})^{1/3}\}$. The claimed bound~\eqref{eq:bound-auto-nonquad} follows with~$c$ replaced by~$\min\{c_1/3, \delta/3\}$.
\end{proof}

\bibliography{expsum-auto}
\bibliographystyle{amsalpha2}

\end{document}